\documentclass[a4paper,12pt]{article}

\usepackage{fullpage}
\usepackage{amsthm,amsmath, amssymb}
\usepackage{palatino,comment}

\input{xy}
\xyoption{all}
\xyoption{matrix}

\newtheorem{theorem}{Theorem}[section]
\newtheorem{teo}[theorem]{Theorem}
\newtheorem{lemma}[theorem]{Lemma}

\newtheorem{lem}[theorem]{Lemma}

\newtheorem{prop}[theorem]{Proposition}

\newtheorem{coro}[theorem]{Corollary}

\theoremstyle{definition}

\newtheorem{defi}[theorem]{Definition}

\newtheorem{ex}[theorem]{Example}
\newtheorem{exs}[theorem]{Examples}

\newtheorem{rem}[theorem]{Remark}

\DeclareMathOperator{\Hom}{Hom}

\DeclareMathOperator{\End}{End}

\newcommand{\wh}{\widehat}

\def\Hom{\mathrm{Hom}}

\def\g{\mathfrak g}
\def\f{\mathfrak f }
\def\l{\mathfrak l}
\def\h{\mathfrak h}
\def\n{\mathfrak n}
\def\z{\mathfrak z}
\def\z{\mathfrak z}

\def\DD{\mathbb D}

\def\ad{\mathrm{ad}}
\def\om{\omega}

\def\dwzz{\delta_{W}^{\Lambda^2\z}}

\def\wzn{(W\wedge \z)^{\n}}
\def\dz{\delta_{\z}}

\def\wg{\wedge}

\def\al{\alpha}
\def\lial{\lambda_{i,\alpha}}
\def\leal{\lambda_{e,\alpha}}

\author{Marco A. Farinati\thanks{Member of CONICET. Partially supported by
PIP 11220110100800CO, and UBACYT 20021030100481BA, mfarinat@dm.uba.ar.} 
\ and Alejandra Patricia Jancsa\thanks{Partially supported by 
PIP 11220110100800CO and UBACYT 20021030100481BA,
pjancsa@dm.uba.ar }
}

\begin{document}
\title{Lie bialgebra structures on 2-step nilpotent graph algebras}
\maketitle

\begin{abstract}
We generalize a result on the Heisenberg Lie algebra that gives restrictions to possible
Lie bialgebra cobrackets on 2-step nilpotent algebras with some additional properties.
For the class of 2-step nilpotent Lie algebras coming from graphs, we
describe these extra properties in a very easy graph-combinatorial way. We
exhibit applications for $\f_n$, the free 2-step nilpotent Lie algebra.
\end{abstract}
\section*{Introduction}

A {\bf Lie bialgebra} is  a triple $(\g,[-,-],\delta)$ where
 $(\g,[-,-])$ is a Lie algebra        and   $\delta:\g\to\Lambda^2\g$
is a linear map such that 
\begin{itemize}
 \item  
$\delta$
satisfies the co - Jacobi identity.
In Sweedler type notation, if $\delta(x)=x_1\wedge x_2$ (sum understood), 
then co-Jacobi condition reads
\[\delta(x_1)\wedge x_2-x_1\wedge \delta(x_2)=0\in \Lambda^3\g\]
 \item $\delta$ satisfies the 1-cocycle
condition $ \delta[x,y]=[\delta x,y]+[x,\delta y]\in \Lambda^2\g$.
\end{itemize}
Lie bialgebras first appear as classical limit of quantum objects, studying
deformations of Hopf algebras and the Quantum Yang-Baxter equation, presented by Drinfel'd.
They also appear in
geometry as Poisson Lie structures on Lie group, and in algebraic topology as 
additional algebraic structure on the space of
(free homotopy classes of) curves in oriented surfaces.
Lie bialgebra structures with underlying complex  simple x Lie algebras were classified
by Belavin and Drinfel'd \cite{BD} under a non-degenerated assumption
 (called factorisable case)
and after that a lot of people work on semi simple and  reductive case. In \cite{FJ} we studied 
the situation of a Lie algebra $\l =\g\times V$ (where $V$ is an abelian factor) and 
the problem of determine all 
possible Lie bialgebra structures supported by $\l$ in terms of Lie bialgebra structures
 on $\g$, and we obtain in that way non factorisable examples for reductive Lie algebras.
Although its importance, the classification problem of Lie bialgebras remain a wide open 
problem. In the opposite side of simple or semi simple case (e.g. solvable or nilpotent) 
there almost no results,  and in some cases is 
hopeless for obvious reasons: if the underlying algebra $\g$ is abelian (say $\dim\g<\infty$),
 to classify all Lie bialgebra structures on  $\g$ is the same as the classification of all
 Lie algebra structures on the vector space $\g^*$.
Surprisingly, if $\g=\h_{2m+1}$, the 2m+1-dimensional Heisenberg Lie algebra,
all Lie bialgebra structures on $\h_{2m+1}$ are known from a long time
(see \cite{BS} and \cite{SZ}), but for general nilpotent Lie algebras there is nothing else.
The nilpotent case is of  special importance 
because after  a result of
Etingof and Gelaki
(see  \cite{EG}, we thank Milen Yakimov for pointing out this
reference), connected Hopf algebras of finite Gelfand Kirillov dimension 
are in bijection with  nilpotent Lie bialgebras.

The origin of this work is the following observation: if $\n$ is a 2-step nilpotent 
Lie algebra with center $\z$, then one can find a linear complement $W$ so that 
$\n=\z\oplus W$. Once we use this (vector space) decomposition, one can see analogies
both in our previous work ($\l=\g\times V$) and in the Heisenberg case 
$\h_{2m+1} =W\oplus k z$. Using this point of view we
obtain, under mild assumptions,  very strong restrictions on all possible co-bracket on a
general 2-step nilpotent algebra. Although these restrictions,
 some nonlinear equations remains to be solved and
the general situation is still wide, but restricted to graph algebras the 
problem is more tractable. Graphs algebras is a very interesting class of 2-step nilpotent
algebras; even if in each dimension there are only a finite number of them, 
they are numerous enough
to construct interesting examples and counterexamples in geometry
(see \cite{LW} and \cite{GGI}).

 The paper is organized as follows:
in section 1 we study general properties of 2-step nilpotent Lie algebras. The main results 
are Theorem \ref{teoTST}, where the Heisenberg case is generalized to arbitrary
 2-step algebras, and Theorem \ref{teoconstruccion} and its reciprocate
\ref{teoCoro}, where the analogy with our previous work \cite{FJ}
is used. In section 2 we specialize to graph algebras. The main results of section 2 are
Theorems \ref{teoLambda2} and
\ref{TSTgraph} where the two algebraic nice hypothesis needed in section 1
 are translated in terms of the combinatorics of the graph.
Section 3 deals with an application of a special class Lie bialgebras, the ones whose
co-bracket  annihilates the center (these class largely include the coboundary ones).
Inside this class we study  $\f_n$, the free 2-step nilpotent algebra on $n$ generators,
that in terms of  graph corresponds to the complete graph $K_n$.

{\bf Acknowledgements:}
We would like to thank Leandro Cagliero for several discussion on 2-step nilpotent algebras
and his help on clarifying equations at the begining of the project. We would also like to thank
 Moira Chas and 
 Dennis Sullivan  for a generous invitation to 
 Stony Brook where we finish the
second part of this work.

\section{Two-step nilpotent Lie bialgebras}
A lie algebra $\n$ is called 2-step nilpotent if
$[\n,[\n,\n]]=0$, or equivalently if $[n,\n]\subseteq\z$, where $\z$ is the center of
$\n$.
As vector space, the center admits a linear complement; we choose one and call it $W$, so 
that
$\n=W\oplus \z$ (as linear vector spaces).
We will study Lie cobrackets $\delta$ defined on the underlying Lie algebra $\n$.
If  $(\n,\delta)$ is a Lie bialgebra structure, $\delta:\n\to\Lambda^2\n$, 
hence it admits a decomposition  via
\[
\n=W\oplus \z \to \Lambda^2(W\oplus\z)=\Lambda^2 W\oplus W\wedge\z\oplus\Lambda^2\z
\]
So, we may split $\delta$ in several components. In particular,
define $\delta_1:=\delta|_W^{\Lambda^2 W}$, that is, for $v\in W$,
$\delta_1(v)$ is the 
${\Lambda^2 W}$ - component of $\delta(v)$.

A very important and classical example is the Heisenberg Lie algebra
 $\n=\h_{2n+1}=W\oplus kz$ where $W=k^{2n}$ with symplectic form
$\om$ and bracket given by
\[
[v+\lambda z,w+\mu z]=\om(v,w)z
\]
For the Heisenberg Lie algebra, all possible cobrackets are well-known
\begin{teo}(\cite{SZ} and \cite{BS})
If  $\delta(z)=v_0\wedge z$ then 
\[
\begin{array}{rccc}
\delta(v)&=&\frac14(\om(v_0,v)\wh\om+v_0\wedge v)& \in\Lambda^2W\\
&&+D(v)\wedge z&\in W\wedge z
\end{array}\]
where $D$ is a coderivation w.r.t. $\delta_1$.
\end{teo}
In particular, $\delta_1$ {\bf is determined} by $\delta(z)$, and 
if $\delta(\z)=0$ then $\delta(v)=D(v)\wedge z\ \forall v\in W$.
We will generalize one direction of the above theorem. We begin with a simple remark
\begin{rem} If $\delta:\n\to \Lambda^2\n$ is a 1-cocycle
then
$\delta(\z)\subset (\Lambda^2\n)^\n$.
\end{rem}
\begin{proof} For $x\in \n, z\in\z$, 
$\ad _z\delta x=[z,x_1]\wedge x_2+x_1\wedge [z,x_2]=0$,
hence
\[
0=\delta (0)=\delta [x,z]=\ad _x\delta z-\ad _z\delta x= \ad _x\delta z+0
\]
That is,  $\delta z$ is ad$_x$-invariant for any $x\in\n, z\in\z$.
\end{proof}

\begin{coro} If $(\Lambda^2\n)^\n=\Lambda^2\z$ then
$\delta (\z)\subset\Lambda^2\z$.
\end{coro}
Even if the Heisenberg Lie algebra does not satisfy
 $(\Lambda^2\n)^\n=\Lambda^2\z$, several interesting
2-step nilpotent Lie algebras do,  as we sill show in next section.
Now the important simplification will be given by  next theorem. We need to introduce
some notation.

For $v,w\in W$, $[v,w]\in\z$, fix a basis $\{z_i\}_{i=1}^{\dim \z}$, and write 
$$[v,w]=\sum _i T_i(v)(w)z_i$$
where $T_i:W\to W^*$ is determined by the Lie bracket.

 Denote by $[\ ,\ ]^*:W^*\times W^*\to W^*$ the dual bracket, i.e. if
$\delta _1(v)=v_1\wedge v_2$ then  $[\lambda,\mu]^*(v)=\lambda(v_1) \mu(v_2)$.
If $\{\lambda_j\}_{j=1}^{\dim W}$ is a basis of $W^*$, 
define $S_j:W^*\to W\cong W^{**}$ so that
\[
 [\lambda,\mu]^*=\sum_jS_j(\mu)(\lambda) \lambda_j
\]

\begin{teo}\label{teoTST}  Let $(\n,\delta)$ be a 2-step nilpotent Lie bialgebra,  then
\begin{itemize}
\item $\delta_1$ is a Lie coalgebra structure on $W$.

\item Consider
$\{\lambda_j\}_{j=1}^{\dim W}$,  $\{\z_i\}_{i=1}^{\dim \z}$ basis of $W^*$ and $\z$, resp. 
and  $T_i$ and $S_j$  as before. If in addition $\delta(\z)\subset\Lambda^2\z$,
then 
\[
T_iS_jT_k+T_kS_jT_i=0 \ \forall i,j,k
\]
\end{itemize}
\end{teo}
For instance, in the Heisenberg Lie algebra there is only on $T$, corresponding to the symplectic form $\om$, which is non degenerate, so the equation
$TS_jT=0 \ \forall j$ clearly implies $\delta_1=0$.

\begin{proof} We need to check co-Jacobi for $\delta_1$.
If $z\in \z$ then $\delta(z)\in(\Lambda^2\n)^\n$ and one can check that we always have
\[
(\Lambda^2\n)^\n\subseteq W\wedge\z\oplus\Lambda^2\z
\]
that is, $\delta(z)$ do has zero component in $\Lambda^2 W$.
Now if
if $v\in W$, $\delta(v)\in\Lambda^2W\oplus W\wedge \z\oplus\Lambda^2\z$ and we can write
\[
\delta (v)=\delta_1(v)\hbox{ Mod }(W\wedge\z+\Lambda^2\z)
\]
and the component of $\delta(v)$ in $\Lambda^3 W$ can only arise because of $\delta_1$
applied to components of $\delta_1(v)$, since $\delta$ applied to elements in $\z$
give components with at least one "z", that is, components in
$\Lambda^2W\wedge \z + W\wedge\Lambda^2\z+\Lambda^3\z$.

The second part is the most interesting, and this part is a generalization of 
the arguments for the Heisenberg case.

Let $u,v\in W$. Since $\n$ is  2-step nilpotent then  $[u,v]\in\z$, but also because $\delta(\z)\subseteq(\Lambda^2\n)^\n$ and we assume
$(\Lambda^2\n)^\n=\Lambda^2\z$ we have
$\delta[u,v]\in\Lambda^2\z $. But also
\[
\delta[u,v]= [\delta u,v]+ [u,\delta v]
\]
Let us write
 \[
\delta u=
\delta_1 u+\sum_i D^i(u)\wedge z_i +\dwzz  u
\]
(and similarly for $\delta v$) where $\dwzz  u$ is the component in $\Lambda^2\z$
of $\delta u$  and
$D^i:W\to W$ are linear maps describing the $W\wedge\z$-component. Then
$\delta[u,v]=$
\[
= [\delta_1 u+\sum_i D^i(u)\wedge z_i +\dwzz  u,v]+ [u,\delta_1 v+\sum_i D^i(v)\wedge z_i+\dwzz  v]
\]
\[
= [\delta_1 u+\sum_i D^i(u)\wedge z_i ,v]+ [u,\delta_1 v+\sum_i D^i(v)\wedge z_i]
\]

\[
= [\delta_1 u,v]+\sum_i [D^i(u),v]\wedge z_i+ [u,\delta_1 v]+\sum_i [u,D^i(v)]\wedge z_i
\]
\[
=\underbrace{ [\delta_1 u,v]+ [u,\delta_1 v]}_{\in W\wedge \z}
+\sum_i\underbrace{ ([D^i(u),v]+[u,D^i(v)])\wedge z_i}_{\in\Lambda^2\z}
\]
  hence
$
\left\{
\begin{array}{rcl}
0&=&[\delta_1 u,v]+ [u,\delta_1 v]\\
\delta[u,v]&=&\sum_i ([D^i(u),v]+[u,D^i(v)])\wedge z_i
\end{array}
\right.
$

In particular, the first equality holds.
The second identity is not needed now, but it will be used in the proof of the
"general construction" theorem.
In Sweedler - type notation,  $\delta_1u=u_1\wedge u_2$, $\delta_1v=v_1\wedge v_2$,   
we have
\[
0=[u_1,v]\wedge u_2+u_1\wedge [u_2,v]+ [u,v_1]\wedge v_2+v_1\wedge [u,v_2]
\]
\[=
-u_2\wedge [u_1,v]+u_1\wedge [u_2,v]-v_2\wedge [u,v_1]+v_1\wedge [u,v_2]\in W\wedge \z
\]
by antisymmetry of $\delta_1$, we have $v_1\wedge v_2=-v_2\wedge v_1$, so
\[
0=
2u_1\wedge [u_2,v]+2v_1\wedge [u,v_2]\in W\wedge \z
\]
Hence $\forall\ \phi\in\z^*$ we have the cocycle formula
\begin{equation}
\label{cociclo2step}
u_1\phi( [v,u_2])=v_1\phi( [u,v_2])
\end{equation}
Let us introduce the following notation;
 for any $v\in W$ and $\phi\in \z ^*$, denote $v_\phi\in W ^*$ by
\[
v_\phi(w):=\phi[v,w]
\]
In this notation, cocycle formula 
reads
 \[
u_1v_\phi( u_2)=v_1u_\phi(v_2)
\]
For any $w$, apply $w_\phi\in W^*$ and get
 \[
w_\phi(u_1)v_\phi( u_2)=w_\phi(v_1)u_\phi(v_2)
\]
 Recall again $[\ ,\ ]^*:=\delta ^*:W^*\times W^*\to W^*$ is the 
transpose of$\delta_1$, so we re-write the cocycle formula as
 \[
[w_\phi,v_\phi]^*( u)=[w_\phi, u_\phi]^*(v)
\]
Now we use alternatively the antisymmetry of $[\ ,\ ]^*$ and the above formula
and get
 \[
[w_\phi,v_\phi]^*( u)
=
-[v_\phi,w_\phi]^*( u)
=
-[v_\phi,u_\phi]^*( w)
=
[u_\phi,v_\phi]^*( w)
\]
\[=
[u_\phi,w_\phi]^*(v)
=
-[w_\phi,u_\phi]^*(v)
\]
hence
\[
\fbox{$[w_\phi,u_\phi]^*(v)=0$ for all $u,v,w\in W$, 
$\phi\in\z ^*$}
\]
This is almost the end of the proof. Now we simply write this formula using bases.
Recall
\begin{align*}
[\cdot,\cdot]^*&=\sum_{k=1}^{\dim W} \lambda_k\otimes S_k,&& S_i:W^*\to W,\quad S+S^t=0; \\    
[\cdot,\cdot]&=\sum_{i=1}^{\dim\z}  z_i\otimes T_i,&& T_i:W\to W^*,\quad T+T^t=0.
\end{align*}
where $\{z_i\}_{i=1}^{\dim\z}$ is  basis of $\z$ and
$\{\lambda_k\}_{i=k}^{\dim W}$ is basis of $W^*$ \\
Besides
$ u_\phi=\sum_{i=1}^{\dim\z}  \phi(z_i) T_i(u)
$, so
\begin{align*}0=
[u_\phi,v_\phi]^*
&=\sum_{i=1}^{\dim\z}\sum_{j=1}^{\dim\z}\phi(z_i)\phi(z_j)[T_i(u),T_j(v)]^*
 \end{align*}
\begin{align*}&=\sum_{i=1}^{\dim\z}\sum_{j=1}^{\dim\z}\sum_{k=1}^{\dim W} 
\phi(z_i)\phi(z_j)\, T_j(v)\big(S_k(T_i(u))\big)\,\lambda_k  
\end{align*}

$T_j(\!a\!)(\!b\!)\!=\!-\!T_j(\!b\!)(\!a\!)
 \Rightarrow T_j(v)(\!S_k(\!T_i(\!u)\!)\!)\!=\!-\!T_j(S_k(\!T_i(u)\!)\!)(v)$, so
\begin{align*}&=-\sum_{i=1}^{\dim\z}\sum_{j=1}^{\dim\z}\sum_{k=1}^{\dim W} 
\phi(z_i)\phi(z_j)\,  T_j\big(S_k(T_i(u))\big)(v)\,\lambda_k,
\end{align*}
Since $\{\lambda_k\}$ is a basis
\[
0=\sum_{i,j} 
\phi(z_i)\phi(z_j)\,  T_j\big(S_k(T_i(u))\big)(v)  \ \forall k\]
Because it holds for all $u$ and $v$ we have
\[0=\sum_{i,j} 
\phi(z_i)\phi(z_j)   T_j S_k T_i \in \Hom(W,W^*)
\]
In particular, if $\{z^i\}$ denotes the dual basis of $\{z_i\}$, 
 taking $\phi=z^{i_0}$ one gets
\[0=   T_{i_0} S_k T_{i_0}\ \forall k,i_0\]
which is a particular case. Specializing at $\phi=z^{i_0}+z^{j_0}$ (and using the parcticular
case) one can easily get
\[\fbox{$0=  T_{j_0} S_k T_{i_0}+ T_{i_0} S_k T_{j_0}\ \forall i_0,j_0,k$}\]
\end{proof}

\begin{rem}
Since the Lie algebra structure will be fixed and the cobracket is the unknown, one can look
at 
\[
"T_iS_jT_k+T_kS_jT_i=0 \ \forall i,j,k"
\]
as a system of linear equations on $S$. This motivates the following definition
\end{rem}

\begin{defi}
Keeping notation for $\n=W\oplus\z$ and the $T_i$'s, we call $\n$ a 2-step Lie algebra 
of {\bf TST} type if
the system of equations
\[
T_iST_k+T_kST_i=0 \ \forall i,k
\]
 with unknown $S:W^*\to W$ has only the trivial solution $S=0$.
\end{defi}
Examples of TST algebras are the Heisenberg algebras. Also
one can check that
$\f_n$, the free 2-step nilpotent Lie algebra on $n$-generators, is of type TST,
and this is included in a big family of
examples that one can build from graphs.

 \begin{coro} If $(\n,\delta)$ is a Lie bialgebra  of type TST
  and $(\Lambda^2\n)^\n=\Lambda^2\z$  then 
\begin{itemize}
\item $\delta\z\subset\Lambda^2\z$
\item  $\delta W\subset W\wedge \z\oplus  \Lambda^2\z$,
\end{itemize}
\end{coro}
\begin{proof}
The first part is because 
$\delta\z\subseteq (\Lambda^2\n)^\n$ and the second is 
because the $\Lambda^2W$-component of $\delta|_W$ is $\delta_1$,
that is determined by $S_j$'s verifying the "TST"  system of equations.
\end{proof}

\subsection{General Construction}

\begin{teo}\label{teoconstruccion}
 Assume the following data on $\n=W\oplus \z$ is given:
\begin{itemize}
\item $\dz:\z\to\Lambda^2\z$ a Lie coalgebra structure,
\item a Lie algebra map  $\DD\colon\z^*\to\End(W)$,
$ f\mapsto  -\sum_if(z_i)D^i
$
verifying the following  
\[
\sum_iT_i(x)(y) \dz z_i
=\sum_{i,j}\Big( T_i(D^j(x))(y)+T_i(x)(D^j(y))\Big)z_i\wg z_j
\]
where $D^i=-\DD(z^i)$.

\item $\Phi\colon \Lambda^2\z^*\to W^*$ a 2-cocycle with values in $W^*$
\end{itemize}
then the map $\delta:\n\to\Lambda^2\n$ defined by
\[
\left\{
\begin{array}{rcllccc}
\delta(z)&=&\dz(z)&\hbox{if } z\in\z
\\
\delta(v)&=&\sum_i D^i(v)\wg z_i
+\Phi^*(v)& \hbox{if }v\in W
\end{array}
\right.
\]
is a  Lie bialgebra structure on $\n$.
\end{teo}
\begin{proof}
Straightforward checking.
\end{proof}
Now TST condition is a useful tool because one can easyly prove the following

\begin{teo}\label{teoCoro}
 If $\n$ is of type TST and $(\Lambda^2\n)^\n=\Lambda^2\z$ then 
{\bf all} Lie bialgebra structures on $\n$ are as in the previous Theorem.
\end{teo}
\begin{proof}
Last corollary says that necesarily
$\delta\z\subset\Lambda^2\z$ and $\delta W\subset W\wedge \z\oplus  \Lambda^2\z$.
If one write $\delta$ in terms of arbitary maps $\dz$, $D^j$'s and $\Phi$, then
it is a straightforward checking that co-Jacobi condition
implies first and third item of Theorem \ref{teoconstruccion} and the 1-cocycle
condition implies the second item.
\end{proof}

\section{Graph algebras}
Let $G=(V,A)$ be an oriented simple graph without loops.
The graph algebra $\n(G)$ associated to a graph $G$ is defined in the following way:
 for each $i,j\in V$
and $\alpha\in A$ going from $i$ to $j$ we set
\[
[e_i,e_j]:=\alpha
\]
Since the isomorphism class do not depends on the orientation 
(just change $\alpha$ by $-\alpha$), 
sometimes we will assume that $G$ is unoriented, but the set of
vertices is ordered, so that an edge joining two vertices can be oriented,
considering that it goes from the smaller to the bigger.

\begin{exs}
\begin{enumerate}
\item the graph
$\xymatrix{
x\ar[r]^Z&y
 }$ gives the 3-dimensional Lie algebra with basis $\{x,y,z\}$ and bracket $[x,y]=z$, that is,
the Heisenberg algebra $\h_3$. On the other hand, for $n>1$ the Heisenberg Lie algebra
$\h_{2n+1}$ is {\em not} a graph algebra.
\item $\f_n$: the free 2-step nilpotent Lie algebra is the graph algebra associated to
the complete graph $K_n$.
\item The nontrivial brackets of the Lie algebra associated to the following graph are
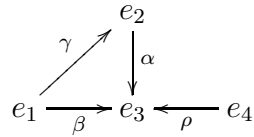
\[
\xymatrix{
&e_2\ar[d]^\alpha\\
e_1\ar[r]_\beta\ar[ru]^\gamma&e_3&e_4\ar^\rho[l]
 }
\]
\[
\begin{array}{cc}
{}[e_1,e_2]=\gamma,&[e_1,e_3]=\beta,\\
{}[e_2,e_3]=\alpha,&[e_4,e_3]=\rho.
\end{array}
\]
\end{enumerate}
\end{exs}
Notice that the center of a graph algebra has natural basis consisting on the
arrows and the isolated vertices. If we assume that the graph does not have isolated vertices
then 
we have canonical decomposition and basis:
\[
W=\bigoplus_{e\in V} ke, \ \z= \bigoplus_{\alpha\in A} k\alpha
\]
and $\n(G)=W\oplus\z$ is a 2-step nilpotent  with center $\z$
and linear complement $W$.

We begin by characterizing the condition $(\Lambda^2\n)^\n=\Lambda^2\z$.
Recall that for a vertex $e$,  the degree - or valency- $|e|$ is the number
 of edges incident to $e$.

\begin{lem}  Consider a graph $G=(V,A)$ 
such that there exists $e\in V$ with $|e|=1$ then there exists an element $0\neq \om\in (W\wedge \z)^\n$. 
\end{lem}

\begin{proof}
Let $\alpha$ be the unique edge joining $e$ with $e'$,
$
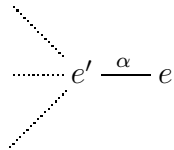
\xymatrix@-1ex{
&\\
&\ar@{..}[ul]\ar@{..}[l]\ar@{..}[dl]e'\ar@{-}[r]^\alpha &e\\
& }
$\\
If $\om:=e\wedge \alpha\in W\wedge \z$ then
$
\ad_{e'}\om=\pm \alpha\wedge\alpha=0
$, and clearly also $\ad_{e''}\om=0$ for any other $e''\in  V$.
\end{proof}
A more involve prove is needed for the following Lemma:
\begin{lem} Let $\om\in \wzn$, if $\om=\underset{e\in V,\al_i\in A}{\sum}\leal e\wg \al_i$, then $\leal =0$
 for all  $\al\in A$ and for each $e\in V$ such that $|e|\geq 2$.
 \end{lem}

\begin{proof}
Consider  $e\in V$ a vertex with $|e|\geq 2$.
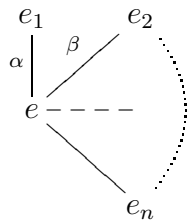
\[
\xymatrix{
e_1&e_2\ar@/^3ex/@{..}[dd]\\
e\ar@{-}[u]^\alpha\ar@{-}[ur]^{ \beta}\ar@{-}[dr]\ar@{--}[r]&\\
&e_n
}
\]
Denote $e'=e_1$ and $e''=e_2$ in the drawing, i.e. $e'$ and $e''$ are two different vertices incident to $e$ with corresponding edges $\al$ and $\beta$.
A general element in $W\wedge\z$
will be of the form
\[
\begin{array}{rcccc}
\om &=&ae\wg\al&+be\wg\beta&+e\wg 
(\underset{\al_i\neq\al,\beta}{\sum}\lambda_i \al_i)
\\
&&+a'e'\wg\al&+b'e'\wg\beta&+e'\wg 
(\underset{\al_i\neq\al,\beta}{\sum}\mu_i \al_i)
\\
&&+a''e''\wg\al&+b''e''\wg\beta&+e''\wg 
(\underset{\al_i\neq\al,\beta}{\sum}\nu_i \al_i)
\\
&&+\underset{e'''\neq e,e',e''}{\sum}
a_{e'''}
e'''\wg\al
&+
\underset{e'''\neq e,e',e''}{\sum}
b_{e'''}e'''\wg\beta
&+
\underset{\underset{\al_i\neq\al,\beta}{e'''\neq e,e',e''}}{\sum}\lambda_{i,e'''}
e'''\wg  \al_i
\end{array}\]
for some $a,a',a''a_i,\mu_i,\nu_i,\lambda_{i,e'''}\in k$.
Since $\om$ is supposed to be invariant, we have
$$
\begin{array}{rcccc}
0=\ad_e\om &=&0&+0&+0\\
&&+a'\al\wg\al&+b'\al\wg\beta&+\al\wg 
(\underset{\al_i\neq\al,\beta}{\sum}\mu_i \al_i)
\\
&&+a''\beta\wg\al&+b''\beta\wg\beta&+\beta\wg 
(\underset{\al_i\neq\al,\beta}{\sum}\nu_i \al_i)
\\
&&+\underset{e'''\neq e,e',e''}{\sum}
a_{e'''}
[e,e''']\wg\al
&+
\underset{e'''\neq e,e',e''}{\sum}
b_{e'''}[e,e''']\wg\beta
&+
\underset{\underset{\al_i\neq\al,\beta}{e'''\neq e,e',e''}}{\sum}\lambda_{i,e'''}
[e,e''']\wg  \al_i
\end{array}$$
then, in particular $b'=a''\quad (*)$. \\
Analogously, compute $0=\ad_{e'}\om=$
$$\begin{array}{rcccc}
 &=&-a\al\wg\al&-b\al\wg\beta&-\al\wg 
(\underset{\al_i\neq\al,\beta}{\sum}\lambda_i \al_i)
\\
&&+0&+0&+0
\\
&&+a''[e',e'']\wg\al&+b''[e',e'']\wg\beta&+[e',e'']\wg (\underset{\al_i\neq\al,\beta}{\sum}\nu_i \al_i) 
\\
&&+\underset{e'''\neq e,e',e''}{\sum}
a_{e'''}
[e',e''']\wg\al
&+
\underset{e'''\neq e,e',e''}{\sum}
b_{e'''}[e',e''']\wg\beta
&+
\underset{\underset{\al_i\neq\al,\beta}{e'''\neq e,e',e''}}{\sum}\lambda_{i,e'''}
[e',e''']\wg  \al_i
\end{array}$$
$$\begin{array}{rcccc}&&=
&-b\al\wg\beta&-\al\wg 
(\underset{\al_i\neq\al,\beta}{\sum}\lambda_i \al_i)
\\
\\
&&+a''[e',e'']\wg\al&+b''[e',e'']\wg\beta&+[e',e'']\wg (\underset{\al_i\neq\al,\beta}{\sum}\nu_i \al_i) 
\\
&&+\underset{e'''\neq e,e',e''}{\sum}
a_{e'''}
[e',e''']\wg\al
&+
\underset{e'''\neq e,e',e''}{\sum}
b_{e'''}[e',e''']\wg\beta
&+
\underset{\underset{\al_i\neq\al,\beta}{e'''\neq e,e',e''}}{\sum}\lambda_{i,e'''}
[e',e''']\wg  \al_i
\end{array}
\quad (1)
$$
Hence, $b=0$. In the same way $0=\ad_{e''}\om =$
$$\begin{array}{rcccc}
&=& -a\beta\wg\al
&&-\beta\wg 
(\underset{\al_i\neq\al,\beta}{\sum}\lambda_i \al_i)
\\
\\
&&+a'[e'',e']\wg\al&+b'[e'',e']\wg\beta&+[e'',e']\wg (\underset{\al_i\neq\al,\beta}{\sum}\nu_i \al_i) 
\\
&&+\underset{e'''\neq e,e',e''}{\sum}
a_{e'''}
[e'',e''']\wg\al
&+
\underset{e'''\neq e,e',e''}{\sum}
b_{e'''}[e'',e''']\wg\beta
&+
\underset{\underset{\al_i\neq\al,\beta}{e'''\neq e,e',e''}}{\sum}\lambda_{i,e'''}
[e'',e''']\wg  \al_i
\end{array}
\quad (2)
$$
Hence, $a=0$.
To continue, we need to consider the following cases $i)$ and $ii)$.

\

$i)$ Suppose $e'$ and $e''$ are not joined by any edge in the graph, so $[e',e'']=0$, then 
$(1)$ equals
$$\begin{array}{rccccccc}0&=
&&\left(
\underset{\al_i\neq\al,\beta}{\sum}\lambda_i \al_i
+\underset{e'''\neq e,e',e''}{\sum}
a_{e'''}
[e',e''']
\right)\wg\al\\
&&
&+
\underset{e'''\neq e,e',e''}{\sum}
b_{e'''}[e',e''']\wg\beta
&+
\underset{\underset{\al_i\neq\al,\beta}{e'''\neq e,e',e''}}{\sum}\lambda_{i,e'''}
[e',e''']\wg  \al_i
\end{array}
\quad (1')
$$
Notice in the previous equation that all the terms belong to different components,
so $(1')$  implies   $\lambda_i=0$ for all $i$ except those corresponding to $\al_i$ incident to $e'$.

In the same way, $(2)$ implies   $\lambda_i=0$ for all $i$ except those corresponding to $\al_i$ incident to $e''$. 
But there is no edge between $e'$ and $e''$, hence all $\lambda_i=0$.

\

$ii)$  Suppose $e'$ and $e''$ are joined by an edge $\gamma$, so $[e',e'']=\gamma$. Equation 1 
reads in this case
$$\begin{array}{rcccccccc}0&=
&&\left( 
\underset{\al_i\neq\al,\beta}{\sum}\lambda_i \al_i+a''\gamma
+\underset{e'''\neq e,e',e''}{\sum}
a_{e'''} [e',e''']
\right)\wg\al
\\
&&&+b'' \gamma \wg\beta&+\gamma \wg (\underset{\al_i\neq\al,\beta}{\sum}\nu_i \al_i) 
\\
&&
&+
\underset{e'''\neq e,e',e''}{\sum}
b_{e'''}[e',e''']\wg\beta
&+
\underset{\underset{\al_i\neq\al,\beta}{e'''\neq e,e',e''}}{\sum}\lambda_{i,e'''}
[e',e''']\wg  \al_i
\end{array}
\quad (1'')
$$
Looking at the terms with common factor $\al$, we see that    $\lambda_i=0$ for all $i$ except those corresponding to $\al_i$ incident to $e'$. A similar computation interchanging $e'$ with $e''$
says $\lambda_i=0$ for all $i$ except those corresponding to $\al_i$ incident to $e''$. 

Resuming, if we call $c=b'=a''$ and $\lambda=\lambda_\gamma$, $\om$ is of the following form
$$\begin{array}{rcccc}
\om &=&&
&e\wg\lambda \gamma
\\
&&+a'e'\wg\al&+ce'\wg\beta&+e'\wg 
(\underset{\al_i\neq\al,\beta}{\sum}\mu_i \al_i)
\\
&&+ce''\wg\al&+b''e''\wg\beta&+e''\wg 
(\underset{\al_i\neq\al,\beta}{\sum}\nu_i \al_i)
\\
&&+\underset{e'''\neq e,e',e''}{\sum}
a_{e'''}
e'''\wg\al
&+
\underset{e'''\neq e,e',e''}{\sum}
b_{e'''}e'''\wg\beta
&+
\underset{\underset{\al_i\neq\al,\beta}{e'''\neq e,e',e''}}{\sum}\lambda_{i,e'''}
e'''\wg  \al_i
\end{array}$$
We go back to equations $(1)$ and $(2)$, we have

$$\begin{array}{rcccc}0&=
&&-\lambda\al\wg  \gamma
\\
\\
&&+c\gamma\wg\al&+b''\gamma\wg\beta&+\gamma\wg (\underset{\al_i\neq\al,\beta}{\sum}\nu_i \al_i) 
\\
&&+\underset{e'''\neq e,e',e''}{\sum}
a_{e'''}
[e',e''']\wg\al
&+
\underset{e'''\neq e,e',e''}{\sum}
b_{e'''}[e',e''']\wg\beta
&+
\underset{\underset{\al_i\neq\al,\beta}{e'''\neq e,e',e''}}{\sum}\lambda_{i,e'''}
[e',e''']\wg  \al_i
\end{array}
\quad (1)
$$
and
$$\begin{array}{rcccc}
0 &=& &&-\lambda\beta\wg  \gamma
\\
\\
&&-a'\gamma\wg\al&-c\gamma\wg\beta&-\gamma\wg (\underset{\al_i\neq\al,\beta}{\sum}\nu_i \al_i) 
\\
&&+\underset{e'''\neq e,e',e''}{\sum}
a_{e'''}
[e'',e''']\wg\al
&+
\underset{e'''\neq e,e',e''}{\sum}
b_{e'''}[e'',e''']\wg\beta
&+
\underset{\underset{\al_i\neq\al,\beta}{e'''\neq e,e',e''}}{\sum}\lambda_{i,e'''}
[e'',e''']\wg  \al_i
\end{array}
\quad (2)
$$
We look at the terms with $\al\wg\gamma$ in $(1)$ and   $\beta\wg\gamma$ in $(2)$ and obtain
$$
\lambda+c=0 \ \hbox{  and }
\lambda-c=0
$$
Hence, $ \lambda=0$, so $\lambda_{e,\al_i}=0$ for all $\al_i\in A$.
\end{proof}
As a corollary we can prove the following characterization:

\begin{teo}\label{teoLambda2}
For a graph algebra (without isolated vertices), $(\Lambda^2\n)^\n
=\Lambda^2\z$ if and only if $|e|\geq 2$ for all $e\in V$.
\end{teo}

\subsection{Graph algebras and the TST equations}

For a graph $(G=(V,A)$ with vertices $V$ and arrows $A$
we use the canonical basis of the center 
$\{z_i\}_{i=1}^{\dim\z}=\{\alpha\}_{\alpha\in \in A}$.
The system of equations with unknown antisymmetric map $S:W^*\to W$
is of the form
\[
T_\al ST_\beta +T_\beta ST_\al=0 \ \forall \al,\beta\in A
\]
Fix $V=\{e_1,\dots,e_n\}$ the set of vertices, it is a basis of $W$ by definition; let 
$\{e_1^*,\dots,e_n^*\}$ be the dual basis. It is easy to see that for each $\alpha\in A$ joining $e_i$ with $e_j$, with $i<j$ we have
$$
T_\al(e_i)=e_j^*,\ T_\al(e_j)=-e_i^* ,\ T_\al(e_k)=0\ \forall k\neq i,j
$$
In matrix notation $[T_\al]=E_{j,i}-E_{i,j}$.
The following is a translation of the TST condition in graph language:

\begin{teo}\label{TSTgraph}
 Let $\n$ be a 2-steps nilpotent Lie algebra arising from a graph 
$G=(V,A)$, $i,j\in V$ and $S$ an antisymmetric solution of the system
$T_\al ST_\beta +T_\beta ST_\al=0 \ \forall \al,\beta\in A$.
\begin{enumerate}
\item If there exists an edge $\al$ joining $i$ and $j$, then $S_{i,j}=0$.

\hskip 4cm $
\xymatrix{
i\ar[r]^\alpha&j}$

\item If there are two edges $\al$ and $\beta$ and four vertices $i,i',j,j'$ such that $\al$ joins $i$ and $i'$ and $\beta$ joins $j$ and $j'$ with $\{i,i'\}\cap\{j,j'\}=\emptyset$ then $S_{i,j}=0$.
\vskip -4ex
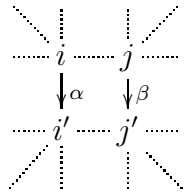
\[
\xymatrix@-2.2ex{
&&&\\
&i\ar[d]^\alpha\ar@{..}[r]\ar@{..}[u]\ar@{..}[l]\ar@{..}[lu]&
j\ar[d]^\beta\ar@{..}[r]\ar@{..}[ur]\ar@{..}[u]&\\
&i'\ar@{..}[r]\ar@{..}[dl]\ar@{..}[l]\ar@{..}[d]&
j'\ar@{..}[r]\ar@{..}[rd]\ar@{..}[d]&\\
&&&
}
\]
\end{enumerate}

\end{teo}

\begin{proof} Consider the equation $T_\al ST_\al=0 $ for a given $\al\in A$, where $\al$ joins $e_i$ with $e_j$.
\[
(T_\al ST_\al)_{i,j}=\sum_{k,\ell} (T_\al)_{i,k} (S)_{k,\ell}(T_\al)_{\ell,j}
=(T_\al)_{i,j} (S)_{j,i}(T_\al)_{i,j}=(S)_{j,i}
\]
This shows 1.
Next suppose we have two edges $\al$ and $\beta$ such that $\al$ joins $i$ and $i'$, 
and $\beta$ joins $j$ and $j'$ with $\{i,i'\}\cap\{j,j'\}=\emptyset$ then $S_{i,j}=0$. 
Writing the TST- equation for
 these $\al,\beta$,
\[
(T_\al ST_\beta +T_\beta ST_\al)_{i',j'}=
\sum_{k,\ell} 
(T_\al)_{i',k} (S)_{k,\ell}(T_\beta)_{\ell,j'}
+
\sum_{k,\ell}  (T_\beta)_{i',k} (S)_{k,\ell}(T_\al)_{\ell,j'}
\]\[= (T_\al)_{i',i} (S)_{i,j}(T_\beta)_{j,j'}
=\pm (S)_{i,j}
\]
since in the second sum, $(T_\beta)_{i',k} =0= (T_\al)_{\ell,j'}$ for all $k,\ell$.
\end{proof}

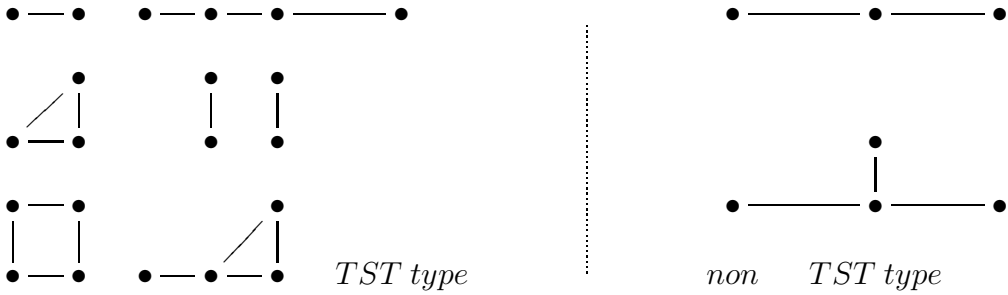
\[
\xymatrix@-2ex{&&&&&&\hbox{Examples}\\
\bullet\ar@{-}[r]&\bullet&\bullet\ar@{-}[r]&\bullet\ar@{-}[r]&\bullet\ar@{-}[r]&\bullet&\ar@{..}[dddd]&\bullet\ar@{-}[r]&\bullet\ar@{-}[r]&\bullet\\
&\bullet\ar@{-}[ld]\ar@{-}[d]&&\bullet\ar@{-}[d]&\bullet\ar@{-}[d]&&&\\
\bullet\ar@{-}[r]&\bullet&&\bullet&\bullet&&&&\bullet\\
 \bullet\ar@{-}[d]\ar@{-}[r]&\bullet\ar@{-}[d]&&&\bullet\ar@{-}[ld]\ar@{-}[d]&&&\bullet\ar@{-}[r]&\bullet\ar@{-}[r]\ar@{-}[u]&\bullet\\
 \bullet\ar@{-}[r]&\bullet&\bullet\ar@{-}[r]&\bullet\ar@{-}[r]&\bullet
&TST\ type &&non&TST\ type}
\]
\begin{rem} 
If we are interested in algebras with $(\Lambda^2\n)^\n=\Lambda^2\z $
then we must look at graphs with $|e|\geq 2$ for all $e\in V$. But if this is the case,
 then  any   pair $i,j$ in $V$
 satisfies either condition 1 or 2 of the previous proposition, so we have the following corollary:
\end{rem}

\begin{coro}
If $|e|\geq 2$ for all $e\in V$ then $(\Lambda^2\n)^\n=\Lambda^2\z$
and also $\n$ is of TST type, as a consequence, any cobracket structure on $\n$ 
satisfies
\[
\delta\z\subseteq\Lambda^2\z,\ \
\delta W\subseteq W\wedge \z\oplus  \Lambda^2\z,
\]
and it is of the form as in Theorem \ref{teoconstruccion}.
\end{coro}

\begin{rem}
In the case where a graph has (all vertices with) valency bigger or equal to 2,
even if the above corollary shows a big simplification of the structure constants, in order
to find all possible data as in Theorem \ref{teoconstruccion} we still have to solve
nonlinear equations, for instance, the data includes a Lie algebra structure on $\z^*$.
 We present next a family of Lie bialgebras that include the ones where the
2-cocycle $\delta$ is a coboundary where this part of the data is trivial.
\end{rem}

\section{Nearly Coboundary Lie bialgebras}
Recall that a Lie bialgebra  $(\n,\delta)$
is called {\bf coboundary} if there exists 
$r\in \Lambda^2\n$ such that
\[
\delta(x)=\ad_x(r)\ \forall x\in \n
 \]
In any coboudary  bialgebra we have $\delta(\z)=0$. This motivates the following
definition
\begin{defi}
A Lie bialgebra $(\n,\delta)$ will be
called
{ \bf nearly coboundary}  if  $\delta |_\z\equiv0$.
\end{defi}

\begin{ex} Let $\delta$ satisfying $\delta(W)\subset\Lambda^2\z$, 
that is, it is of the form  
\[\delta(e_i)=
\underset{\al,\beta\in A}{\sum}\mu^{\al,\beta}_i\al\wg\beta\quad \forall i\in  V
\]
with arbitrary coefficients $\mu^{\al,\beta}_i\in K$ verifing 
 only $\mu^{\al,\beta}_i=-\mu^{\beta,\al}_i$ for all $i,\al,\beta$. Then
  $\delta$ endows $\n$ of a  nearly coboundary Lie bialgebra structure. 
   Notice that there are $|V|\frac{|A|(|A|-1)}{2}$ free parameters, while for
    a coboundary
Lie bialgebra structure with $\delta(W)\subset \Lambda^2\z$ we 
need an element
$$r=\sum_{i\in V,\al\in A}r_{ij}e_i\wg\al\in W\wg\z$$
In order to give $r$ we need $|V||A|$ parameters, and for 
$|A|>3$ it is clear that
$|V|\frac{|A|(|A|-1)}{2}>|V||A|$, so in particular there is a lot non coboundary Lie algebras in this
family of examples.

As particular cases, for $G=C_n$  we have $|V|=n$ and $|A|=n$ and 
and $G=K_n$, $|V=n$ and $|A=n(n-1)/2$. Except $n=3$, there are
a lot of nearly coboundary Lie bialgebras of this type that are not coboundary.
 In the table we write the numbers
$|V|\frac{|A|(|A|-1)}{2}$ and $|V||A|$ in general and for small $n$:

\[\begin{array}{|c|cc||cc|}
\hline&C_n&&K_n&\\
&|V||A|&|V|\frac{|A|(|A|-1)}{2}&|V||A|&|V|\frac{|A|(|A|-1)}{2}\\
\hline
n&n^2&\frac{n^2(n-1)}{2}& \frac{n^2(n-1)}{2}&\frac18n^2(n+1)(n-1)(n-2)\\
\hline 3&9&9&9&9 \\
\hline 4&16&24&24&60\\
\hline 5&25&50&50&225\\
\hline 6&36&90&90&630\\
\hline
\end{array} 
\]
\end{ex}
Now we will study conditions of Theorem \cite{teoconstruccion} for nearly coboundary
bialgebra structures on graph algebras $\n=\n(G)$.

\begin{lemma} Let $G$ be a graph with $|e|\geq 2\ \forall e\in V$, 
$\n=\n(G)$,  $\delta:\n\to\Lambda^2\n$ a Lie bialgebra structure. Assume
that $\delta$ is nearly coboundary, so that $\delta(z)=0$ if $z\in \z$.
If we write  (as in Theorem \ref{teoconstruccion})
\[
\delta(v)=\sum_{\alpha\in A}  D_\alpha (v)\wg \alpha 
+\Phi^*(v) \hbox{ if }v\in W
\]
then  $D_\al D_\beta=D_\beta D_\al$ for all edges $\alpha,\beta$.
\end{lemma}
\begin{proof}
We know that all bialgebra structures are as in Theorem  \ref{teoconstruccion}, so
they must satisfy
\[
\underset{\al,\in A}{\sum}D_\al(e_i)\wg \delta(\al)
=\underset{\al,\beta\in A}{\sum}D_\al(D_\beta(e_i))\wg \al\wg\beta
\]
But because $\delta(\z)=0$ we have
\[
0=\underset{\al,\beta\in A}{\sum}D_\al(D_\beta(e_i))\wg \al\wg\beta=
\underset{\al <\beta\in A}{\sum}(D_\al D_\beta-D_\beta D_\al)(e_i)\wg \al\wg\beta
\]
So $D_\al D_\beta=D_\beta D_\al$.
\end{proof}

From the above lemma, we see that a typical situation will be when the $D_\alpha$'s are simultaneously diagonalizable. In next subsection we study this particular case.

\subsection{Nearly coboundary bialgebras with diagonalizable $D_\alpha$'s}

In this subsection
we suppose that all the $\{D_\al:\al \in A\}$ are simultaneously diagonalizable, but moreover
that the set of vertices $\{e_i:1\leq i\leq n\}=V$  is a basis of eigenvectors.
We denote by $\lial$ the corresponding eigenvalues, i.e., $\forall i,\al$
\[D_\al (e_i)=\lial e_i\]

\begin{rem}
The assumption that the set of vertices are eigenvalues is not
a lost of generality in the family of free 2-step nilpotent algebras.
\end{rem}

\begin{prop}\label{proplambda} Let $\al_0$ be an edge joining $i_0$ with $j_0$, then
$$\lambda_{i_0,\al}=-\lambda_{j_0,\al}\ \forall \al\neq \al_0$$
\end{prop}

\begin{proof} From the cocycle condition, the assumption $\delta(\z)=0$ and
$\delta_W(e_{i_0})=\underset{\al\in A}{\sum}D_\al(e_{i_0})\wg\al
$, we have
$$0=\delta(\al_0)=\delta([e_{i_0},e_{j_0}])
=[\delta(e_{i_0}],e_{j_0}]+[e_{i_0},\delta(e_{j_0})]$$
$$= \left[\underset{\al\in A}{\sum}D_\al(e_{i_0})\wg\al, e_{j_0}
\right]
+ \left[e_{i_0},\underset{\al\in A}{\sum}D_\al(e_{j_0})\wg\al
\right]
$$
$$= \left[\underset{\al\in A}{\sum}\lambda_{i_0,\al}e_{i_0}\wg\al, e_{j_0}
\right]
+ \left[e_{i_0},\underset{\al\in A}{\sum}\lambda_{j_0,\al}e_{j_0}\wg\al
\right]
$$
$$=\al_0\wg \left(\underset{\al\in A}{\sum}\lambda_{i_0,\al}+\lambda_{j_0,\al}
\right)\al
$$
then each coefficient $\lambda_{i_0,\al}+\lambda_{j_0,\al}=0$ for all $\al\neq\al_0$.
\end{proof}

\begin{ex}For the graph 
$
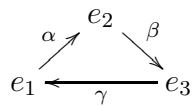
\xymatrix@-2ex{
&e_2\ar[rd]^{\beta}\\
e_1\ar[ru]^{\alpha}&&e_3\ar[ll]^{\gamma}
 }
$

we get  $f_3$, the free 2-step nilpotent Lie algebra on 3 generators 
$\{e_1,e_2,e_3\}$. The complete list of Lie bialgebra 
structures on $\f_3$ such that $\delta (\z)=0$ and  diagonal
 $D_{\alpha _i}$ for all $i=1,2,3$ is 
\[
\delta (e_1)=e_1\wg(a\alpha +b\beta +c\gamma) + \omega _1 =:e_1\wg A_1+\om_1
\]
\[
\delta (e_2)=e_2\wg(a\alpha -b\beta -c\gamma) +\omega _2=:e_2\wg A_2+\om_2
\]
\[
\delta (e_3)=e_3\wg(-a\alpha -b\beta +c\gamma) + \omega _3 =:e_3\wg A_3+\om_3
\]
for any $a,b,c\in k$, $\omega _i\in\Lambda ^2\z$ satisfying co - Jacobi condition:
\[\om _i\wg A_i=0,\ \ i=1,2,3\]
\end{ex}
For highly connected graphs the situation is even more favorable:

\begin{prop} Consider a graph $G=(V,A)$, a vertex
$i_0\in V$ and $\alpha \in A$. If there exists a path in $G$ of the forms\\

\hskip 2cm $
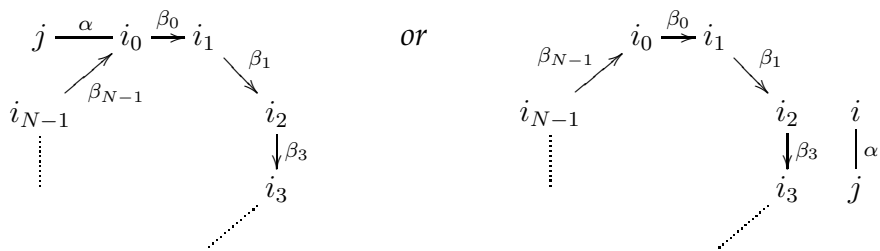
\xymatrix@-2ex{
j&\ar@{-}[l]_ \alpha i_0\ar[r]^{\beta_0}&i_1\ar[rd]^{\beta_1}&\\
 i_{N-1}\ar[ru]_{\beta_{N-1}}&   &      &i_2\ar[d]^{\beta_3}\\
    \ar@{..}[u]&&&i_3\ar@{..}[dl]\\
&&&
}$ \hskip 1cm or \hskip 1cm 
$
\xymatrix@-2ex{
& i_0\ar[r]^{\beta_0}&i_1\ar[rd]^{\beta_1}&\\
 i_{N-1}\ar[ru]^{\beta_{N-1}}&   &      &i_2\ar[d]^{\beta_3}&i\ar@{-}[d]^ \alpha\\
    \ar@{..}[u]&&&i_3\ar@{..}[dl]&j\\
&&&
}$

then $\lambda_{i_0,\alpha}=(-1)^N\lambda_{i_0,\alpha}$; if
$N$ is odd then $\lambda_{i_0,\alpha}=0$
\end{prop}

As a corollary we have

\begin{coro} \label{propfn} If $n\geq 4$ and $\delta$ is a bialgebra structure   on  $\f_n$
satisfying $\delta(\z)=0$ and $D_\al$  diagonalizable for all $\al$, then
$$\delta(W)\subset\Lambda^2\z,$$
 that is, all $D_\al$ are necessarily zero and
$\delta(e_i)$ is of the form $\delta(e_i)=\om_i$
with   arbitrary $\om_i \in \Lambda^2\z$.
\end{coro}

\begin{proof}if $i_0$ is a vertex and $\alpha$ an edge, in the complete graph on $n$ vertices with $n\geq 4$
we are always in one of the following situations:

\

$
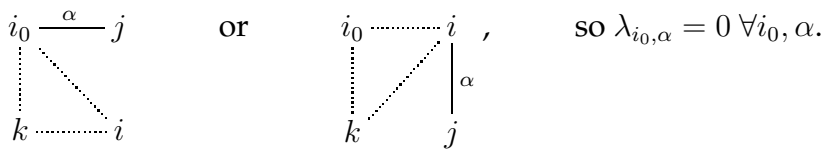
\xymatrix{
i_0\ar@{..}[d]\ar@{..}[dr]\ar@{-}[r]^{\alpha}&j\\
k\ar@{..}[r]&i}$
\hskip 1cm or   \hskip 1cm 
$
\xymatrix{
i_0\ar@{..}[r]\ar@{..}[d]&i\ar@{-}[d]^{\alpha}\\
k\ar@{..}[ur]&j
}$, \hskip 1cm 
so $\lambda_{i_0,\alpha}=0$ $\forall i_0,\alpha$.

\end{proof}

We finish exhibiting examples of non diagonalizable $D_i$'s in $\f_3$. We do not know
if there are similar nontrivial examples in $\f_n$ for $n\geq 4$.

\subsection{Non-diagonalizable $D$'s in $\f_3$}
We fix the notation
$
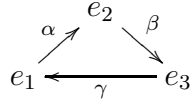
\xymatrix@-2ex{
&e_2\ar[rd]^{\beta}\\
e_1\ar[ru]^{\alpha}&&e_3\ar[ll]^{\gamma}
 }
$

Since $[D_{\alpha_i},D_{\alpha_j}]=0$ we may always assume that there exists a basis where all $D_i$ are simultaneously upper triangular, that is, of the form
$\left(\begin{array}{ccc}
a&b&c\\
0&d&e\\
0&0&f
\end{array}\right)$.

But also, if one of them has three different elements in the diagonal, 
then all of them are diagonalizable, and we are in the previous case. So
we may assume that all of them have some multiplicity on the eigenvalues.
We slit in two cases: multiplicity 2 and multiplicity 3.
In multiplicity 3 there are two possibilities:
there is a $D$ with a single Jordan block 
of size $3$, or the maximal size of the Jordan block is 2.
Writing the first $D$ in Jordan form and using that the other $D$'s commute
with this one we arrive at the following cases:

Multiplicity 2: 
\[
D_\alpha=\left(\begin{array}{ccc}
\lambda&1&0\\
0&\lambda &0\\
0&0&\lambda'
\end{array}\right),\
D_\beta=\left(\begin{array}{ccc}
a&b&0\\
0&a&0\\
0&0&c
\end{array}\right),\
D_\gamma=\left(\begin{array}{ccc}
\mu&\nu&0\\
0&\mu&0\\
0&0&\tau
\end{array}\right)\hskip 1cm (I)
\]
where $\lambda\neq\lambda'$, and
 multiplicity 3:
\[
D_\alpha=\left(\begin{array}{ccc}
\lambda&1&0\\
0&\lambda&1\\
0&0&\lambda
\end{array}\right),\
D_\beta=\left(\begin{array}{ccc}
a&b&c\\
0&a &b\\
0&0&a
\end{array}\right),\
D_\gamma=\left(\begin{array}{ccc}
\mu&\nu&\rho \\
0&\mu&\nu\\
0&0&\mu
\end{array}\right)\hskip 1cm (II)
\]
or
\[
D_\alpha=\left(\begin{array}{ccc}
\lambda&1&0\\
0&\lambda&0\\
0&0&\lambda
\end{array}\right),\
D_\beta=\left(\begin{array}{ccc}
a&b&c\\
0&a &0\\
0&0&a
\end{array}\right),\
D_\gamma=\left(\begin{array}{ccc}
\mu&\nu&\rho \\
0&\mu&0\\
0&0&\mu
\end{array}\right)\hskip 1cm (III)
\]
From the matrix parameters introduce the elements $A,B,C,D\in W$
\[
A:= \lambda\alpha+a\beta+\mu \gamma,\
B:= \alpha +b\beta +\nu\gamma,\
C:= c\beta +\rho\gamma,\
D:=\lambda'\alpha+c\beta+\tau\gamma\]

Take elements $\om_i\in\Lambda^2\z$ ($i=1,2,3$)
and define
\[
\delta(e_i):=D_\alpha(e_i)\wedge\alpha+
D_\beta(e_i)\wedge\beta+
D_\gamma(e_i)\wedge\gamma+
\om_i\]
In all cases we have
\[\begin{array}{rcl}
\delta(e_1)&=&e_1\wedge A+ \om_1\\
\delta(e_2)&=&e_2\wedge A+e_1\wedge B+ \om_2
\end{array}
\]
and $\delta(e_3)$, depending on cases, is equal to
\[\delta(e_3)=
\left\{
\begin{array}{llc}
&e_3\wedge D+ \om_3&(I)\\
&e_3\wedge A +e_2\wedge B+e_1\wedge C+ \om_3&(II)\\
&e_3\wedge A +e_1\wedge C+ \om_3&(III)\\
\end{array}
\right.\]
The restriction
given by the co-Jacobi identity are
\[
\om_1\wedge A =0=\om_2\wedge A + \om_1\wedge B
\]
and, depending on cases
\[
\begin{array}{ccl}
(I)&0=&\om_3\wedge D\\
(II)&0=&\om_3\wedge A+\om_ 2\wedge B+\om_1\wedge C \\
(III)&0=&\om_3\wedge A+\om_1\wedge C
\end{array}
\]

\end{document}